\documentclass[11pt,a4paper]{article}
\usepackage[margin=28mm]{geometry}
\usepackage{amsmath,amssymb,amsthm,mathtools,booktabs,array,microtype,xcolor,hyperref,fancyhdr}
\hypersetup{colorlinks=true,linkcolor=blue!45!black,citecolor=blue!45!black,urlcolor=blue!45!black}
\newtheorem{definition}{Definition}[section]\newtheorem{lemma}[definition]{Lemma}\newtheorem{proposition}[definition]{Proposition}\newtheorem{theorem}[definition]{Theorem}\newtheorem{corollary}[definition]{Corollary}\newtheorem{conjecture}[definition]{Conjecture}\theoremstyle{remark}\newtheorem{remark}[definition]{Remark}\newtheorem{example}[definition]{Example}
\newcommand{\Part}{\Pi}\newcommand{\kk}{\mathbf{k}}\newcommand{\cN}{\mathcal N}\newcommand{\Hilb}{\operatorname{Hilb}}\newcommand{\rhoj}{\rho}
\title{\textbf{A Categorical Framework for the Negami Polynomial and Its Splitting Formula}\\[2mm]\large State-Sum Lifts, Boundary Partitions, and the Complete Two-Vertex Calculation}
\author{\textbf{IWAO MIZUKAI}\\\small Chiba Keizai University High School\\\small Chiba, Japan\\\small \href{mailto:get_mizukai@hotmail.com}{get\_mizukai@hotmail.com}}
\date{}
\begin{document}\maketitle
\begin{abstract}
We give an elementary categorical lift of the three-variable Negami polynomial $f(G;t,x,y)$ and of the boundary-state structure underlying Negami's splitting formula. To a finite graph we associate a triply graded vector space whose Hilbert polynomial is exactly $f(G;t,x,y)$; deletion--contraction becomes a direct-sum decomposition. For a graph with a distinguished boundary vertex set $U$, we refine the construction by the partition of $U$ induced by each spanning state. We prove an object-level gluing formula and show that its decategorification produces the matrix $T_n(t)=(t^{\rho(\pi,\sigma)})$. For $n=2$ we derive $T_2(t)$ and its inverse $B_2(t)$ explicitly and verify the splitting formula on a four-cycle. Finally, we isolate the obstruction to realizing $B_n(t)$ inside finite-dimensional graded vector spaces and formulate a derived/localized version of the problem. The results proved here constitute an additive categorical framework; construction of a genuinely new Negami homology with a nontrivial differential remains open.
\end{abstract}

\section{Introduction}
Negami introduced in \cite{Negami1987} a three-variable graph polynomial $f(G;t,x,y)$ characterized by deletion--contraction and the value $t^n$ on $n$ isolated vertices. One of its most distinctive features is a splitting formula for $G=K\cup H$ when the two subgraphs meet in a finite vertex set. The formula is controlled by the partition lattice of the common vertex set and by a matrix $T_n(t)$ whose inverse gives universal splitting coefficients. Oxley established the close relation with the Tutte polynomial \cite{Oxley1989}; singular values and generalized inverses of the splitting matrix were later studied by Burgos \cite{Burgos2018}.

Categorification replaces a polynomial invariant by a graded or homological object whose graded dimension or Euler characteristic recovers the polynomial. For graph polynomials, Jasso-Hernandez and Rong constructed a bigraded chain complex whose graded Euler characteristic is a version of the Tutte polynomial and obtained exact sequences reflecting deletion--contraction \cite{JassoRong2006}.

Our aim is structural. We first lift the full three-variable Negami state sum to a triply graded vector space. We then refine this lift for graphs with boundary and prove that the boundary gluing mechanism already exists at the level of graded objects. Thus $T_n(t)$ appears as the decategorification of a categorical pairing matrix. The inverse $B_n(t)=T_n(t)^{-1}$ contains signs and denominators, identifying where a stronger homological theory must go beyond finite-dimensional graded vector spaces. No claim is made that the additive lift alone is a new homology theory.

\section{The Negami polynomial and its state sum}
Throughout, graphs are finite and may have loops and multiple edges. For $S\subseteq E(G)$, let $\omega_G(S)$ be the number of connected components of the spanning subgraph $(V(G),S)$.

\begin{definition}[Negami polynomial]
The Negami polynomial is determined by
\[f(\overline K_n;t,x,y)=t^n,\qquad f(G;t,x,y)=xf(G/e;t,x,y)+yf(G-e;t,x,y),\]
where contraction of a loop is identified with deletion, following \cite{Negami1987}.
\end{definition}

\begin{theorem}[State-sum formula]\label{thm:statesum}
For every finite graph $G$,
\begin{equation}\label{eq:statesum}
\boxed{f(G;t,x,y)=\sum_{S\subseteq E(G)}t^{\omega_G(S)}x^{|S|}y^{|E(G)|-|S|}.}
\end{equation}
\end{theorem}
\begin{proof}
Induct on $m=|E(G)|$. For $m=0$ the only state is empty and gives $t^{|V(G)|}$. For $m>0$, fix $e$. States not containing $e$ are states of $G-e$ and contribute an extra factor $y$. States containing $e$ correspond, after deleting $e$ from the state and contracting it, to states of $G/e$ with the same component count and an extra factor $x$. Hence the state sum satisfies the defining recurrence and the initial condition.
\end{proof}

\begin{corollary}[Relation with the Tutte polynomial]\label{cor:tutte}
Let $c(G)=\omega_G(E(G))$, $r(G)=|V(G)|-c(G)$ and $\nu(G)=|E(G)|-r(G)$. Then
\begin{equation}\label{eq:tutte}
\boxed{f(G;t,x,y)=t^{c(G)}x^{r(G)}y^{\nu(G)}T_G\!\left(1+\frac{ty}{x},1+\frac{x}{y}\right).}
\end{equation}
\end{corollary}
\begin{proof}
Insert $X-1=ty/x$ and $Y-1=x/y$ into the state-sum definition of $T_G(X,Y)$ and use $r(S)=|V(G)|-\omega_G(S)$. After multiplication by the displayed monomial prefactor, every summand becomes the corresponding summand of \eqref{eq:statesum}.
\end{proof}

\section{An additive categorification}
Let $\mathrm{GrVect}^{\mathbb Z^3}_{\kk}$ be the category of finite-dimensional $\mathbb Z^3$-graded vector spaces over a field $\kk$. Write $M\{a,b,c\}$ for grading shift and
\[\Hilb(M)=\sum_{i,j,k}\dim_{\kk}M_{i,j,k}\,t^ix^jy^k.\]

\begin{definition}[Negami state object]
Define
\begin{equation}\label{eq:N0}
\boxed{\cN_0(G)=\bigoplus_{S\subseteq E(G)}\kk\{\omega_G(S),|S|,|E(G)|-|S|\}.}
\end{equation}
\end{definition}

\begin{theorem}[Additive categorification]\label{thm:additive}
The isomorphism class of $\cN_0(G)$ is a graph invariant and
\[\boxed{\Hilb(\cN_0(G))=f(G;t,x,y).}\]
\end{theorem}
\begin{proof}
Each state contributes a one-dimensional homogeneous summand whose Hilbert monomial is precisely its term in \eqref{eq:statesum}. A graph isomorphism induces a bijection on edge subsets preserving both cardinality and component count.
\end{proof}

\begin{theorem}[Object-level deletion--contraction]\label{thm:dc}
For every edge $e\in E(G)$,
\begin{equation}\label{eq:dc}
\boxed{\cN_0(G)\cong\cN_0(G/e)\{0,1,0\}\oplus\cN_0(G-e)\{0,0,1\}.}
\end{equation}
\end{theorem}
\begin{proof}
Partition the states according to whether $e$ is selected. States not containing $e$ are states of $G-e$ with one additional unit in the unselected-edge grading. States containing $e$ correspond to states of $G/e$ with the same component count and one additional unit in the selected-edge grading. Taking the direct sum gives \eqref{eq:dc}.
\end{proof}

\begin{remark}
This is an additive, Hilbert-series categorification. Since no differential is present, it is not yet a homology theory stronger than the polynomial itself.
\end{remark}

\section{Boundary partitions and categorical gluing}
Let $K$ have distinguished boundary $U=\{u_1,\ldots,u_n\}\subseteq V(K)$.

\begin{definition}[Boundary partition]
For $S\subseteq E(K)$, let $\pi_K(S)\in\Part(U)$ be the partition in which $u_i,u_j$ are in the same block precisely when they are connected in $(V(K),S)$.
\end{definition}

\begin{definition}[Generated boundary connectivity]
For $\pi,\sigma\in\Part(U)$, let $\rhoj(\pi,\sigma)$ be the number of blocks in the smallest equivalence relation containing the relations represented by both partitions. Under the refinement order this equals $|\pi\vee\sigma|$. Negami uses the opposite order convention, where the same operation is denoted by a meet.
\end{definition}

\begin{lemma}[Component count under gluing]\label{lem:components}
Suppose $G=K\cup_UH$, where $K$ and $H$ have disjoint edge sets and intersect only in $U$. For states $S_K,S_H$, put $\pi=\pi_K(S_K)$ and $\sigma=\pi_H(S_H)$. Then
\begin{equation}\label{eq:components}
\boxed{\omega_G(S_K\cup S_H)=(\omega_K(S_K)-|\pi|)+(\omega_H(S_H)-|\sigma|)+\rhoj(\pi,\sigma).}
\end{equation}
\end{lemma}
\begin{proof}
Exactly $|\pi|$ components of the $K$-state meet $U$, so $\omega_K(S_K)-|\pi|$ are internal and survive gluing unchanged; similarly for $H$. The boundary-meeting components are identified according to the transitive closure of the two boundary equivalence relations, leaving $\rhoj(\pi,\sigma)$ components. Summing gives the formula.
\end{proof}

\begin{definition}[Boundary-state object]
For $\pi\in\Part(U)$ set
\begin{equation}\label{eq:boundaryobject}
\boxed{\cN_{U,\pi}(K)=\bigoplus_{\substack{S\subseteq E(K)\\\pi_K(S)=\pi}}\kk\{\omega_K(S)-|\pi|,|S|,|E(K)|-|S|\}.}
\end{equation}
\end{definition}

\begin{theorem}[Object-level splitting formula]\label{thm:catsplit}
Under the hypotheses above,
\begin{equation}\label{eq:catsplit}
\boxed{\cN_0(G)\cong\bigoplus_{\pi,\sigma\in\Part(U)}\cN_{U,\pi}(K)\otimes\kk\{\rhoj(\pi,\sigma),0,0\}\otimes\cN_{U,\sigma}(H).}
\end{equation}
\end{theorem}
\begin{proof}
A state of $G$ is uniquely a pair $(S_K,S_H)$. The second and third gradings add under tensor product, and Lemma \ref{lem:components} gives exactly the first grading. Thus the one-dimensional homogeneous summands on the two sides correspond bijectively degree by degree.
\end{proof}

Writing $F_{K,\pi}=\Hilb(\cN_{U,\pi}(K))$, decategorification gives
\begin{equation}\label{eq:partialsplit}
f(K\cup_UH)=\sum_{\pi,\sigma\in\Part(U)}t^{\rhoj(\pi,\sigma)}F_{K,\pi}F_{H,\sigma}.
\end{equation}

\section{Recovery of Negami's splitting matrix}
For $\alpha\in\Part(U)$, let $K/\alpha$ be obtained by identifying all boundary vertices in each block of $\alpha$.

\begin{lemma}[Boundary quotient identity]\label{lem:quotient}
For every $\alpha\in\Part(U)$,
\begin{equation}\label{eq:quotient}
\boxed{f(K/\alpha)=\sum_{\pi\in\Part(U)}t^{\rhoj(\alpha,\pi)}F_{K,\pi}.}
\end{equation}
\end{lemma}
\begin{proof}
Fix a state of boundary type $\pi$. It has $\omega_K(S)-|\pi|$ internal components. After the identifications prescribed by $\alpha$, its boundary-meeting components become $\rhoj(\alpha,\pi)$ components. Edge numbers are unchanged. Summing over all states of type $\pi$, then over $\pi$, gives \eqref{eq:quotient}.
\end{proof}

Order $\Part(U)=\{\pi_1,\ldots,\pi_m\}$ and define
\[
T_n(t)=\bigl(t^{\rhoj(\pi_i,\pi_j)}\bigr)_{i,j},\qquad
\mathbf F_K=(F_{K,\pi_i})_i,\qquad
\mathbf q_K=(f(K/\pi_i))_i.
\]
Lemma \ref{lem:quotient} says $\mathbf q_K=T_n(t)\mathbf F_K$.

\begin{theorem}[Negami splitting formula]\label{thm:negamisplit}
If $T_n(t)$ is invertible, then
\begin{equation}\label{eq:negamisplit}
\boxed{f(K\cup_UH)=\mathbf q_K^{\mathsf T}T_n(t)^{-1}\mathbf q_H.}
\end{equation}
\end{theorem}
\begin{proof}
Equation \eqref{eq:partialsplit} is $f(K\cup_UH)=\mathbf F_K^{\mathsf T}T_n\mathbf F_H$. Since $T_n$ is symmetric and invertible, substitute $\mathbf F_K=T_n^{-1}\mathbf q_K$ and $\mathbf F_H=T_n^{-1}\mathbf q_H$.
\end{proof}

\section{The complete two-vertex calculation}
Let $U=\{u_1,u_2\}$. There are two partitions:
\[C=\{\{u_1,u_2\}\},\qquad D=\{\{u_1\},\{u_2\}\}.\]

\begin{proposition}
In the order $(C,D)$,
\begin{equation}\label{eq:T2}
\boxed{T_2(t)=\begin{pmatrix}t&t\\t&t^2\end{pmatrix}.}
\end{equation}
\end{proposition}
\begin{proof}
Combining $C$ with either partition yields one generated boundary block, while combining $D$ with itself leaves two blocks.
\end{proof}

\begin{proposition}[Explicit inverse]\label{prop:B2}
For $t\neq0,1$,
\begin{equation}\label{eq:B2}
\boxed{B_2(t)=T_2(t)^{-1}=\frac1{t^2(t-1)}\begin{pmatrix}t^2&-t\\-t&t\end{pmatrix}.}
\end{equation}
Equivalently,
\[
B_2(t)=\begin{pmatrix}\frac1{t-1}&-\frac1{t(t-1)}\\[1mm]-\frac1{t(t-1)}&\frac1{t(t-1)}\end{pmatrix}.
\]
\end{proposition}
\begin{proof}
$\det T_2=t^3-t^2=t^2(t-1)$, so the ordinary $2\times2$ inverse formula gives \eqref{eq:B2}.
\end{proof}

\begin{corollary}[Two-vertex splitting]
If $K$ and $H$ meet only in $u_1,u_2$, then
\begin{align}
f(K\cup_UH)={}&\frac{f(K/C)f(H/C)}{t-1}
-\frac{f(K/C)f(H)}{t(t-1)}\\
&-\frac{f(K)f(H/C)}{t(t-1)}
+\frac{f(K)f(H)}{t(t-1)}.\label{eq:2split}
\end{align}
\end{corollary}

\begin{example}[Verification on $C_4$]\label{ex:C4}
Let $K$ be the path $u_1-a-u_2$ and $H$ the path $u_1-b-u_2$. Their union is the four-cycle $C_4$. Directly from the state sum,
\[
f(K)=t^3y^2+2t^2xy+tx^2.
\]
After identifying $u_1$ and $u_2$, the graph $K/C$ has two vertices joined by two parallel edges; hence
\[
f(K/C)=t^2y^2+2txy+tx^2.
\]
The same formulas hold for $H$. Substitution into \eqref{eq:2split} simplifies to
\begin{equation}\label{eq:C4}
\boxed{f(C_4)=t^4y^4+4t^3xy^3+6t^2x^2y^2+4tx^3y+tx^4.}
\end{equation}
This agrees with a direct count: the $0,1,2,3,4$-edge states occur with multiplicities $1,4,6,4,1$, and every proper spanning subgraph of $C_4$ is a forest. Thus the two-vertex splitting formula is verified in this nontrivial example.
\end{example}

\section{The splitting matrix as a categorical Gram matrix}
For $\pi,\sigma\in\Part(U)$ define the one-dimensional pairing object
\[
\Theta_{\pi\sigma}=\kk\{\rhoj(\pi,\sigma),0,0\}.
\]
Then $\Hilb(\Theta_{\pi\sigma})=t^{\rhoj(\pi,\sigma)}$.

\begin{proposition}
The matrix $\Theta_n=(\Theta_{\pi\sigma})$ decategorifies entrywise to $T_n(t)$. Moreover, Theorem \ref{thm:catsplit} is the categorical bilinear pairing
\[
\cN_0(K\cup_UH)\cong\mathbf{\cN}_K^{\mathsf T}\Theta_n\mathbf{\cN}_H,
\]
where matrix multiplication means tensor product in each entry followed by direct sum.
\end{proposition}
\begin{proof}
The first statement follows from the Hilbert polynomial of each $\Theta_{\pi\sigma}$. Expanding the categorical matrix product gives exactly the direct sum in \eqref{eq:catsplit}.
\end{proof}

Thus $T_n(t)$ is naturally interpreted as a Gram-type matrix for boundary connectivity states. Negami's inverse matrix $B_n(t)$ should accordingly be interpreted as an inverse pairing kernel after passage to an enlarged categorical setting.

\section{Why the inverse requires an enlarged category}
\begin{proposition}[Finite-dimensional obstruction]\label{prop:obstruction}
The entries $1/(t-1)$ and $-1/[t(t-1)]$ of $B_2(t)$ cannot be Hilbert polynomials of finite-dimensional $\mathbb Z$-graded vector spaces.
\end{proposition}
\begin{proof}
The Hilbert polynomial of a finite-dimensional graded vector space is a Laurent polynomial with nonnegative integer coefficients and finite support. The displayed rational functions are not Laurent polynomials; one also carries a negative sign. Hence neither can occur as such a Hilbert polynomial.
\end{proof}

\begin{remark}[Homological shifts handle signs]
In the Grothendieck group of a triangulated or dg setting, $[X[1]]=-[X]$. Thus negative coefficients can be encoded by homological shift. The essential additional difficulty is the denominator, especially $(t-1)^{-1}$.
\end{remark}

\begin{proposition}[Formal realization after completion]
In a $t$-adically completed setting allowing locally finite infinite sums,
\[
\frac1{t-1}=-\frac1{1-t}=-(1+t+t^2+\cdots).
\]
Consequently the rational factor $1/(t-1)$ admits a formal Euler-characteristic realization by an infinite graded object together with a homological shift.
\end{proposition}
\begin{proof}
This is the geometric-series identity in $\mathbb Z[[t]]$; the overall minus sign is represented in a triangulated Grothendieck group by a shift. The statement is formal and does not produce a finite object.
\end{proof}

For general $n$, the known determinant formula for Negami's connectivity matrix implies singularities at $t=0,1,\ldots,n-1$; see \cite{Burgos2018}. At such parameters one should expect the categorical pairing itself to become degenerate, suggesting radicals, negligible morphisms, quotient categories, or generalized inverse data rather than an ordinary inverse.

\section{Toward a homological Negami theory}
Equation \eqref{eq:tutte} makes clear that existing Tutte categorifications are relevant. Jasso-Hernandez and Rong construct a bigraded chain complex and exact sequences categorifying a version of deletion--contraction \cite{JassoRong2006}. Regrading and specialization therefore provide homological information associated with the Negami polynomial through its Tutte equivalence. The construction in the present paper is different in emphasis: it keeps all three variables visible and isolates the boundary-partition mechanism of Negami's splitting formula.

A natural next target is a category $\mathcal P_U(t)$ encoding boundary partitions, together with modules $\mathcal M_U(K)$ for boundary graphs, such that gluing is represented by a derived tensor product.

\begin{conjecture}[Derived Negami gluing]\label{conj:derived}
There exists a suitable dg or derived partition-type category $\mathcal P_U(t)$ and an assignment $K\mapsto\mathcal M_U(K)$ for graphs with boundary $U$ such that
\begin{equation}\label{eq:derived}
\boxed{\mathcal M_U(K\cup_UH)\simeq\mathcal M_U(K)\overset{\mathbf L}{\otimes}_{\mathcal P_U(t)}\mathcal M_U(H),}
\end{equation}
and the induced Euler pairing on the Grothendieck group has Gram matrix $T_n(t)$.
\end{conjecture}

For $n=2$, the first concrete problem is to construct, in an appropriate localization or completion, an inverse kernel whose Grothendieck matrix is
\[
\frac1{t^2(t-1)}\begin{pmatrix}t^2&-t\\-t&t\end{pmatrix}.
\]
Such a construction would categorify not merely the pre-inversion boundary pairing but the actual two-vertex splitting coefficients.

\section{Status of the results and conclusion}
The state-sum lift, deletion--contraction decomposition, boundary-state decomposition, object-level gluing theorem, Gram-matrix interpretation, and the complete $n=2$ calculation are proved in this paper. The proposed derived inverse kernel and a nontrivial Negami homology are not proved here; they are research problems motivated by the preceding structure.

The main point is that the Negami polynomial admits a direct additive categorification in which its special splitting mechanism remains visible before decategorification. Boundary connectivity is represented by partition-indexed graded objects, and $T_n(t)$ becomes a categorical pairing matrix. The obstruction in Proposition \ref{prop:obstruction} then explains why the inverse matrix naturally points toward derived localization or completion. This provides a precise starting point for a genuinely homological theory of Negami splitting.

\section*{Acknowledgements}
The author is grateful to the literature on Negami's polynomial and graph-polynomial categorification that motivated the present formulation.

\end{document}